\documentclass[a4paper,10pt]{amsart}
\usepackage[utf8]{inputenc}
\usepackage{amssymb}
\usepackage{amsthm}
\usepackage{mathrsfs}
\newtheorem{thm}{Theorem}[section]
\newtheorem*{thm*}{Theorem}

\newtheorem{lemma}[thm]{Lemma}

\theoremstyle{definition}

\theoremstyle{lemma}
\newtheorem*{lem*}{Lemma}

\numberwithin{equation}{section}

\title{Some singular value inequalities via convexity}
\author{Zoltán Léka}

\address{Royal Holloway, University of London \\ Egham Hill \\ Egham \\ Surrey \\ TW20 0EX \\ United Kingdom}
\email{zoltan.leka@rhul.ac.uk}

\thanks{This study was supported by the Marie Curie IF Fellowship, Project 'Moments', Grant no. 653943.}


\begin{document}
         
 \begin{abstract}
   If $c_1(Z) \geq \hdots \geq c_n(Z)$ denote the Euclidean lengths of the column vectors of any $n \times n$ matrix $Z,$
   then a fundamental inequality related to Hadamard products states that 
        $$ \sum_{i=1}^k \sigma_i(X^*Y \circ B) \leq \sum_{i=1}^k c_i(X) c_i(Y) \sigma_i(B) \qquad 1 \leq k \leq n,$$
   where $\sigma_i(\cdot)$ is the $i$th singular value. In this paper, we shall offer a simple proof of this result via
   convexity arguments. In addition, this technique is applied to obtain some further singular value inequalities as well.
 \end{abstract}

\maketitle
  
  \section{Introduction}
  
   Let us denote the ordered singular values of any $n \times n$ matrix $A$ by $\sigma_1(A) \geq \hdots \geq \sigma_n(A).$ 
   For any $A, B \in M_n(\mathbb{C}),$ von Neumann's trace inequality states that 
    \begin{equation}
      |\mbox{Tr } AB| \leq \sum_{i=1}^n \sigma_i(A)\sigma_i(B).
    \end{equation}
   One may consider the result as a matricial analogue of the rearrangement inequality for reals \cite{HLP}. 
   Relying upon the polar decomposition of matrices, a straightforward corollary is the following:
    \begin{equation}
       \sum_{i=1}^n \sigma_i(AB) \leq \sum_{i=1}^n \sigma_i(A)\sigma_i(B).
    \end{equation}
   However, we know that a family of inequalities is valid here, in fact,
     \begin{equation}
       \sum_{i=1}^k \sigma_i(AB) \leq \sum_{i=1}^k \sigma_i(A)\sigma_i(B) \qquad \mbox{for all } \quad 1 \leq k \leq n,    
     \end{equation}
    (see e.g. \cite{B}). 
 
   Now let $\circ$ denote the Hadamard product of matrices; i.e. $ (A \circ B)_{ij} = a_{ij} b_{ij}.$
   The following singular value inequality holds for Hadamard products as well:
    \begin{equation}
       \sum_{i=1}^k \sigma_i(A \circ B) \leq \sum_{i=1}^k \sigma_i(A)\sigma_i(B) \qquad \mbox{for all } \quad 1 \leq k \leq n.     
    \end{equation}

   We recall that several different proofs of this result can be found in the literature \cite{BS}, \cite{HJ}, \cite{O}, \cite{Z}, furthermore, a unified approach to (1.3) and (1.4) was provided in \cite{HM}.
   Singular value inequalities of Hadamard products have already attracted many researchers, and a common improvement of scattered
   results was given by Ando, Horn and Johnson in \cite{AHJ} and \cite{HJb}. Their result says that if $c_1(Z) \geq \hdots \geq c_n(Z)$ denote the ordered Euclidean lengths of the column vectors of any $n \times n$ matrix $Z,$
   then, for any decomposition $A = X^*Y,$ 
       \begin{equation}
         \sum_{i=1}^k \sigma_i(A \circ B) \leq \sum_{i=1}^k c_i(X) c_i(Y) \sigma_i(B) \qquad 1 \leq k \leq n.
       \end{equation}
   This result is obviously much stronger than $(1.4).$ For a general review and thorough exposition on singular value inequalities, we refer the reader to
   \cite{A}, \cite{B} and \cite{HJb}.
   
   Our goal is to offer a new proof of the previous inequality (1.5) by means of convexity arguments and exploit
   this technique to deduce further inequalities as well. 
   Briefly, the main idea here is to tackle a chain of convex optimization problems over the unit balls
   of matrices with respect to different matrix norms. 
   Then, with an adequate description of the extreme points of these balls, we can readily establish the above inequalities.

   \section{Some weighted norms on $\mathbb{R}^n$ and $\mathbb{M}_n$}

   We say that an $n$-tuple $w \in \mathbb{R}^n_+$ is a {\it weight} if its non-negative entries are arranged decreasingly $w_1 \geq \hdots \geq w_n \geq 0.$
   Given a vector $x \in \mathbb{R}_n,$ denote $|x_i|^\downarrow$ the decreasingly ordered components of its entrywise modulus. 
   For any $x \in \mathbb{R}^n,$ the weighted vector $k$-norm is defined by
    $$ \|x\|_{(k)}^w = \sum_{i=1}^k w_i |x_i|^\downarrow,$$
   where $w$ is a weight and $w_1 \geq \hdots \geq w_k > 0.$ We just simply write $\|\cdot\|_{(k)}$ for the vector $k$-norm when
   $w_1 = \hdots = w_k = 1.$ We notice that the usual $\ell_1^n$ and $\ell_\infty^n$ norms now are denoted by $\|\cdot\|_{(n)}$ and  
   $\|\cdot\|_{(1)},$ respectively. Determining the dual norm $\|\cdot\|_{(k)^*}^w$ of the weighted vector $k$-norm may be a bit tricky.
   For instance, this has been done
   in the earlier papers \cite{WDST}, \cite[Lemma 1]{L}, and particularly in \cite{BPS}. In fact, 
     $$ \|x\|_{(k)^*}^w = 
            \max \left\{ {\| x\|_{(1)} \over w_1}, {\| x\|_{(2)} \over w_1 +  w_2}, \hdots, {\|x\|_{(k-1)} \over w_1 + \hdots + w_{k-1} }, {\| x\|_{(n)} \over w_1  + \hdots + w_k} \right \}.$$
   For simplicity, we recall that $\|\cdot\|_{(k)^*} = \max \: (\|\cdot\|_{(1)}, {1 \over k}\|\cdot\|_{(n)}),$ see \cite{B}.        
   These observations may lead to a description of the extreme points of the unit ball $\mathcal{B}_{(k)}^w$ in $\mathbb{R}^n$         
   with respect to $ \|\cdot\|_{(k)}^w.$ Indeed, we may infer that 
    \begin{equation}
      {\rm ext } \:  \mathcal{B}_{(k)}^w \subseteq \bigcup_{1 \leq j \leq k-1} \left(\sum_{i=1}^j w_i\right)^{-1} E_j  \cup \left(\sum_{i=1}^k w_i \right)^{-1} E_n,
    \end{equation}
   where $E_i$ is the set of vectors in $\mathbb{R}^n$ with $i$ non-zero coordinates which are $+1$ or $-1$ (see \cite[Lemma 2]{L}). 
   A different approach to obtain
   ${\rm ext } \:  \mathcal{B}_{(k)}^w$ is presented in \cite{LT}, however, the inclusion $(2.1)$ is enough for the rest of the paper.   
   
   Given an $n \times n$ matrix $A,$ its ordered singular values are denoted by $\sigma_1(A) \geq \hdots \geq \sigma_n(A).$   
   The weighted Ky Fan $k$-norms of $A$ are defined by 
      $$ \|A\|_{(k)}^w = \sum_{i=1}^k  w_i\sigma_i(A).$$
   We recall that these norms are unitarily invariant; i.e. $\|A\|_{(k)}^w =  \|UAV\|_{(k)}^w$ for any unitary $U$ and $V.$  The set of $n \times n$ partial isometries of rank-$k$ is denoted by $R_k.$ 
   From (2.1) and a simple application of the singular value decomposition, we readily get the inclusion 
     \begin{equation}
      \mbox{\rm ext } \mathfrak{B}_{(k)}^w \subseteq  \bigcup_{1 \leq j \leq k-1} \left(\sum_{i=1}^j w_i\right)^{-1} R_j  \cup \left(\sum_{i=1}^k w_i \right)^{-1} R_n
    \end{equation}
   (see e.g. \cite{LT} and \cite{CLT}). It is somewhat interesting that a very similar result holds for a class of
   unitarily invariant norms defined through the $s$-numbers in infinite dimensional Hilbert
   spaces \cite{CLT}.
   We notice that the extreme points with respect to the Ky Fan $k$-norms are given by 
      \begin{equation}
      {\rm ext } \:  \mathfrak{B}_{(k)} = R_1 \cup R_n/k, 
    \end{equation}
    see \cite[Exercise IV.2.12]{B}.
   
   In the next section we shall need another weighted norm on the linear space of $n \times n$ matrices.
   For any matrix $A,$ let $c_1(A) \geq \hdots \geq c_n(A)$ be the decreasingly arranged Euclidean norms of the column vectors of $A.$
   Then one can consider the weighted norm of $A$ by  
    $$ \|A\|_{c,k}^w = \sum_{i=1}^k w_i c_i(A).$$
   To get a description of the extreme points of the unit ball $\mathfrak{B}_{c,k}^w$ with respect to the norm $\|\cdot\|_{c,k}^w,$
   let $C_k$ denote the set of $n \times n$ matrices which has exactly $k$ non-zero column vectors and each non-zero column vector is a (Euclidean) unit vector.
   Then a simple reasoning gives the inclusion
      \begin{equation}
      \mbox{\rm ext } \mathfrak{B}_{c,k}^w \subseteq  \bigcup_{1 \leq j \leq k-1} \left(\sum_{i=1}^j w_i\right)^{-1} C_j  \cup \left(\sum_{i=1}^k w_i \right)^{-1} C_n.
    \end{equation}
   Indeed, we need to check that if $A \in \mbox{\rm ext } \mathfrak{B}_{c,k}^w$ then
   $c(A) = (c_1(A), \hdots, c_n(A)) \in {\rm ext } \: \mathcal{B}_{(k)}^w.$ If $c(A) = {1 \over 2} (x + y),$
   where $x, y \in  \mathcal{B}_{(k)}^w,$ let $O_i$ denote the rotation of the Euclidean space $\mathbb{R}^n$ such that 
   $O_i((0, \hdots, 0 , c_i(A), 0, \hdots, 0)^T) = i$th row vector of $A$. Now let $X$ and $Y$ denote the matrices such that 
   their $i$th rows are given by $O_i((0, \hdots, 0 , x_i, 0, \hdots, 0)^T)$ and $O_i((0, \hdots, 0 , y_i, 0, \hdots, 0)^T),$
   respectively. Then it is simple to see that $A = {1 \over 2} (X + Y)$ and $X, Y \in  \mathfrak{B}_{c,k}^w.$ 
   Furthermore, we notice that the extreme points with respect to the  $\|\cdot\|_{c,k}$ norms (i.e. $w = (1,\hdots, 1))$ are given by 
      \begin{equation}
       {\rm ext } \:  \mathfrak{B}_{c,k} = C_1 \cup C_n/k
      \end{equation}

     \section{Proof of the Ando--Horn--Johnson inequality}

     We start with two preliminary lemmas.
     We notice that these inequalities essentially include the extremal cases in (1.5). Throughout the paper we say that
     $A \in M_n(\mathbb{C})$ is a contraction if $\sigma_1(A) \leq 1.$ 
    
     \begin{lemma}
       Let $X, Y$ be $n \times n$ matrices such that $c_i(X), c_i(Y) \leq 1$ for all $1 \leq i \leq n.$ If $S$ is contraction, 
       then 
        $$ \sigma_1(X^* Y\circ S) \leq 1. $$
     \end{lemma}
     
     \begin{proof}  First, assume that $S$ is unitary: $S \equiv U = (u_{ij}).$
       Then we need to check that
        $$\|(X^* Y\circ U)z\|_2 \leq \|z\|_2$$
       for all $z \in \mathbb{C}^n.$ Indeed, setting $X^* = (\underline{x}^*_1 \: \hdots \: \underline{x}^*_n)^T$ and $Y=(\underline{y_1} \: \hdots \: \underline{y}_n),$
       \begin{align*}
          \|(X^* Y\circ U)z\|_2^2 &= \sum_{i=1}^n \Bigl| \sum_{j=1}^n \langle \underline{x}^*_i, \underline{y}_j \rangle u_{ij} z_j \Bigr|^2 \\
                                  &\leq \sum_{i=1}^n  \|\underline{x}^*_i\|_2^2 \Bigl\| \sum_{j=1}^n   u_{ij} z_j \underline{y}_j  \Bigr\|^2_2 \\
                                  &\leq  \sum_{i=1}^n \sum_{k=1}^n \Bigl| \sum_{j=1}^n   u_{ij} z_j (\underline{y}_j)_k  \Bigr|^2 \\
                                  &=  \sum_{k=1}^n \left( \sum_{i=1}^n \Bigl| \sum_{j=1}^n   u_{ij} z_j (\underline{y}_j)_k  \Bigr|^2 \right) \\
                                  &=  \sum_{k=1}^n  \Biggl( \sum_{j=1}^n  | z_j (\underline{y}_j)_k |^2 \Biggr) \\
                                  &= \sum_{j=1}^n  | z_j|^2 \|\underline{y}_j\|_2^2 \\
                                  &\leq \|z\|_2^2.
       \end{align*}
         Since every contraction $S$ is the convex sum of unitaries, the proof is complete.
     \end{proof}

     We recall that Lemma 3.1 was settled in \cite[p. 6]{JN} with a completely different method.
    
   \begin{lemma}
       Let $X, Y$ be $n \times n$ matrices such that $c_i(X) = c_i(Y) = 1$ for all $1 \leq i \leq n.$ If $Q$ is a rank-one partial isometry, 
       then 
        $$ \sum_{i=1}^n \sigma_i(X^* Y\circ Q) \leq 1. $$
     \end{lemma}
     \begin{proof}
        Let $u$ and $v$ be unit vectors such that $Q = uv^*.$ 
        Then $$X^*Y \circ uv^* = ((\overline{u} \otimes {\bf 1}) \circ X)^* ((\overline{v} \otimes {\bf 1}) \circ Y).$$
        Notice that $\tilde{X} =  (\overline{u} \otimes {\bf 1}) \circ X$ and $\tilde{Y} =  (v \otimes {\bf 1}) \circ Y$ have Hilbert--Schmidt norms at most $1,$
        because the Euclidean norm of all entries is at most $1.$
        Thus von Neumann's trace inequality implies
        \begin{align*}
            \sum_{i=1}^n \sigma_i(X^* Y\circ Q) = \sum_{i=1}^n \sigma_i(\tilde{X}^* \tilde{Y})  &\leq  \sum_{i=1}^n \sigma_i(\tilde{X}^*) \sigma_i( \tilde{Y}) \\
                                                       &\leq  \left(\sum_{i=1}^n \sigma_i^2(\tilde{X})\right)^{1/2} \left(\sum_{i=1}^n \sigma_i^2( \tilde{Y})\right)^{1/2} \\
                                                       &\leq 1.
        \end{align*}
    \end{proof}

   \begin{thm}
      Let $A = X^*Y$ and $B$ be $n \times n $ matrices. Then 
       $$ \sum_{i=1}^k \sigma_i(A \circ B) \leq \sum_{i=1}^k c_i(X) c_i(Y) \sigma_i(B)$$
      holds for all $1 \leq k \leq n.$ 
   \end{thm}
      
    \begin{proof}
       From a continuity argument we may assume that $c_i(X)$ and $c_i(Y)$ are non-zero reals.
       Let $c(Z)$ denote the vector $(c_1(Z), \hdots, c_n(Z))$ for any $Z \in M_n(\mathbb{C}).$
       Notice that we need to prove $$\max \left\{ \|X^*Y \circ B\|_{(k)} \colon B \in \mathfrak{B}_{(k)}^{c(X)c(Y)} \right\} = 1,$$
       for any fixed $X$ and $Y.$ Since the objective function is convex in $B,$
       we have the following two cases relying on (2.4).
       
       {\noindent \it Case 1.} $B = (\sum_{i=1}^j c_i(X)c_i(Y))^{-1} B_j,$ where $B_j$ is a rank-$j$ partial isometry and $1 \leq j \leq k-1.$
       We check that 
         $$ \sum_{i=1}^k \sigma_i(X^*Y \circ B_j) \leq \sum_{i=1}^j c_i(X) c_i(Y),$$ 
       or equivalently, 
            $$\max \left\{ \|X^*Y \circ B_j\|_{(k)} \colon X \in \mathfrak{B}_{c,j}^{c(Y)} \right\} = 1$$
       for any fixed $Y$ and $B_j.$

       First, if $ X = (\sum_{i=1}^l c_i(Y))^{-1}  X_l,$ where $1 \leq l \leq j-1,$ and $X_l \in C_l,$ we have by (2.5) that
         \begin{align*}
           \max \left\{ \|X_l^*Y \circ B_j\|_{(k)} \colon Y \in \mathfrak{B}_{c,l} \right\} &= \max \left\{ \|X_l^*Y \circ B_j\|_{(k)} \colon Y \in C_1 \cup C_n/l \right\} \\
                                                                                            &\leq \max_{X_l,Y \in C_n} \left\{ \|X_l^*Y \circ B_j\|_{(k)} \colon Y \in C_1 \mbox{ or } X_l \in C_1  \right\} \\
                                                                                            &\leq 1,
         \end{align*}
          where the last inequality follows from Lemma 3.1.
         
         Secondly, if $ X = (\sum_{i=1}^j c_i(Y))^{-1} X_n,$ where $X_n \in C_n,$  it is enough to see that
         $$\max \left\{ \|X_n^*Y \circ B_j\|_{(k)} \colon Y \in \mathfrak{B}_{c,j} \right\} = \max \left\{ \|X_n^*Y \circ B_j\|_{(k)} \colon Y \in C_1 \cup C_n/j \right\}   = 1.$$
        Again, if $Y \in C_1$ we can simply apply Lemma 3.1. If $Y = Y_n/j,$ where $Y_n \in C_n,$ note that $B_j/j$
        has trace class norm $1.$ Hence it may be assumed that $B_j/j$ is a rank-one partial isometry and Lemma 3.2 completes the proof.
         
         {\noindent \it Case 2.} $B = (\sum_{i=1}^k c_i(X)c_i(Y))^{-1} U,$ where $U$ is unitary.
         We claim that 
          $$\max \left\{ \|X^*Y \circ U\|_{(k)} \colon X \in \mathfrak{B}_{c,k}^{c(Y)} \right\} = 1.$$
          If $ X = (\sum_{i=1}^l c_i(Y))^{-1}  X_l,$ where $X_l \in C_l$ and $1 \leq l \leq k-1,$ then
          the claim exactly follows from the argument used in Case $1.$ Otherwise, from (2.4) we may assume that $ X \in (\sum_{i=1}^k c_i(Y))^{-1} X_n$ and $X_n \in C_n.$ 
         Now by 2.5 $$\displaystyle \max \left\{ \|X_n^*Y \circ B_j\|_{(k)} \colon Y \in \mathfrak{B}_{c,k} \right\} = \max \left\{ \|X_n^*Y \circ B_j\|_{(k)} \colon Y \in C_1 \cup C_n/k \right\} =  1.$$ 
         Indeed, if $Y \in C_1,$ we can apply Lemma 3.1. Lastly, if $Y = Y_n/k,$ where $Y_n \in C_n,$ we have
        \begin{align*}
            {1 \over k} \sum_{i=1}^k \sigma_i(X_n^* Y_n \circ U) &\leq \sigma_1(X_n^* Y_n \circ U) \leq 1 
        \end{align*}
        by Lemma 3.2.
        
        The proof is complete now.
    \end{proof}
    
    \section{Remarks}
    
   \subsection{von Neumann's trace inequality} We note that an analogue reasoning gives that the only extremal case in the von Neumann inequality $(1.1)$ is when 
   $A = xy^*$ is a rank-one partial isometry and $B$ is unitary. Indeed, we need to maximize
   the convex function $A \mapsto |\mbox{Tr } AB|$ over the unit ball $\mathfrak{B}_{(n)}^{\sigma(B)}$ and consider 
   (2.2) and the equality (2.3). Then we readily get $|{\rm Tr} AB| = |\langle y,B x \rangle| \leq \|y\|_2 \|Bx\|_2 = \sigma_1(B),$ and the proof is complete. 
   Certain convexity arguments were also used in \cite{G} to obtain von Neumann's fundamental result. 
   
   \subsection{Unitarily invariant norms of bilinear forms} We also remark that using the previous technique it is straightforward to get the inequality (1.3) as well.
   However, we leave the proof to the interested reader and now turn to a result of Horn, Mathias and Nakamura \cite{HM}.
   
   For simplicity, let \textbullet \: denote a symmetric bilinear form on $M_n(\mathbb{C}) \times M_n(\mathbb{C}).$ Every
   \textbullet \: defines a bilinear form \textbullet$_R$ by 
    $$ \mbox{Tr} (A \mbox{ \textbullet } B) C = \mbox{Tr} (B \mbox{ \textbullet}_R  C) A. $$
   Then one has the following characterization \cite[Theorem 1]{HM} of singular value inequalities related to \textbullet.
   
   \begin{thm}
     Let $A$ and $B$ $n \times n$ matrices. Then 
     $$ \sum_{i=1}^k \sigma_i(A \mbox{ \textbullet } B) \leq  \sum_{i=1}^k \sigma_i(A) \sigma_i(B)  \quad \mbox{ for all } \quad  1 \leq k \leq n$$
     if and only if 
     $$ \sigma_1(A \mbox{ \textbullet } B) \leq \sigma_1(A)\sigma_1(B) \mbox{ and }   \sigma_1(A \mbox{ \textbullet}_R B) \leq \sigma_1(A)\sigma_1(B). $$
   \end{thm} 

   \begin{proof} '$\Longleftarrow$':
      First, let $\displaystyle A = \left(\sum_{i=1}^j \sigma_i(B)\right)^{-1} A_j,$ where $A_j$ is a rank-$j$
      isometry and $1 \leq j \leq k-1.$ Notice that 
        \begin{align*}
             \max \{ \|A_j  \mbox{ \textbullet } B\|_{(k)} \colon B \in \mathfrak{B}_{(j)}\} 
                 &= \max \{ \|A_j  \mbox{ \textbullet } B\|_{(k)} \colon B = U/j \mbox{ for some unitary } U  \\ 
                 & \hspace{3cm} \mbox{ or } B \mbox{ rank-1 partial isometry}  \} \\
                &= \max \{ \|A_j  \mbox{ \textbullet } B\|_{(k)} \colon A_j \mbox{ or }  B \mbox{ is rank-1 partial isometry}\}. 
        \end{align*}
       But, for any rank-$1$ partial isometry $X$ and contraction $Z,$  
        $$\|X \mbox{ \textbullet } Z\|_{(k)} = \mbox{Tr } X (Z \mbox{ \textbullet }_R Y) \leq \sigma_1(Z \mbox{ \textbullet }_R Y)
         \leq \sigma_1(Z)\sigma_1(Y) \leq 1,$$
       where $Y \in \mathfrak{B}_{(k)^*}.$ 
        
        Lastly, if  $\displaystyle A = \left(\sum_{i=1}^k \sigma_i(B)\right)^{-1} U,$ where $U$ is a unitary, we get
         \begin{align*}
             \max \{ \|U  \mbox{ \textbullet } B\|_{(k)} \colon B \in \mathfrak{B}_{(k)}\} &= \max \{ \|U  \mbox{ \textbullet } B\|_{(k)} \colon B \in R_1 \cup R_n/k \}  \\
                 &=  \max \{ \|U  \mbox{ \textbullet } B\|_{(k)} \colon B \in R_1  \} \\
                     & \qquad \vee  \max \{ \sigma_1(U  \mbox{ \textbullet } B) \colon B \in R_n \}  \\
                 &\leq \max \{ \sigma_1(B)\sigma_1(U  \mbox{ \textbullet }_R X) \colon B \in R_1, X \in \mathfrak{B}_{(k)^*}  \} \\
                     & \qquad \vee  \max \{ \sigma_1(U  \mbox{ \textbullet } B) \colon B \in R_n \}  \\
                 &\leq 1.    
        \end{align*}
       
        The function  $A \mapsto \|A \mbox{ \textbullet } B\|_{(k)}$ is convex on the unit ball $\mathfrak{B}_{(k)}^{\sigma(B)},$ hence with (2.2)
        at hand the proof is complete.
          
         '$\Longrightarrow$': We need to check that $\sigma_1(A \mbox{ \textbullet}_R B) \leq \sigma_1(A)\sigma_1(B)$ follows. 
         In fact, 
          \begin{align*}
             \sigma_1(A \mbox{ \textbullet}_R B)  &= \max \{ \mbox{Tr } (A \mbox{ \textbullet}_R B)X \colon X \in R_1 \} \\
                                                                             &= \max \{ \mbox{Tr } A (B \mbox{ \textbullet} X) \colon X \in R_1 \} \\ 
                                                                             &\leq \sigma_1(A) \sum_{k=1}^n \sigma_k(X \mbox{ \textbullet} B) \\
                                                                             &\leq \sigma_1(A)\sigma_1(B),
          \end{align*}
          where the first inequality comes from von Neumann's inequality. 
         \end{proof}
         
         \subsection{Fan products} Given $A$ and $B$ $n \times n$ matrices, the Fan product is defined by
          $$ (A \star B)_{ij} = \begin{cases}
                                   -a_{ii}b_{ii} &\mbox{if } i=j, \\
                                    a_{ij}b_{ij} &\mbox{otherwise},
                                \end{cases}
                               $$
         see \cite{RH}. Then it is simple to check that Tr $(A \star B)C = \mbox{ Tr } A( B^T \star C).$ Moreover,
         one has the inequalities
          $$ \sigma_1(A \star B) \leq \sigma_1(A) \sigma_1(B) \quad \mbox{ and } \quad  \sigma_1(A^T \star B) \leq \sigma_1(A) \sigma_1(B). $$
          This can be directly proved by the method of \cite{JN}. In fact, for any $x \in \mathbb{R}^n,$ define the map
           $$ \Theta(x) = \mbox{diag} (-x_1, x_2, \hdots, x_n) \in M_n(\mathbb{C}).$$
          If $A = (\underline{a}_1, \hdots, \underline{a}_n)$ and $B = (\underline{b}_1, \hdots, \underline{b}_n),$ let 
            $$ \Phi(A) = (\Theta(\underline{a}_1) | \hdots | \Theta(\underline{a}_n)) $$ and 
            $$ \Psi(B) = \underline{b}_1 \oplus \hdots \oplus \underline{b}_n.$$
           Then $$ A \star B = \Phi(A)\Psi(B).$$
           Additionally, $\Phi$ and $\Psi$ are contractive, hence $\sigma_1(\Phi(A)\Psi(B)) \leq \sigma_1(\Phi(A))\sigma_1(\Psi(B)) \leq \sigma_1(A)\sigma_1(B).$
           Now Theorem 4.1 implies that we have a family of inequalities (1.4) related to the Fan product. Simple counterexamples show
           that Lemma 3.1 does not hold with $\star$, hence (1.5) either: for instance, $X=Y= 1/\sqrt{3} ({\bf 1} \otimes {\bf 1}),$
           where ${\bf 1} = (1,1,1)^T \in \mathbb{R}^3,$ and $U$ is the unitary
                 $$ U = {1 \over 3}\begin{bmatrix}
                                 2 & 1  & 2 \\ 
                                -2 &  2 & 1  \\
                                 1  & 2 & -2 \\
                        \end{bmatrix}.
$$
        
         \subsection{Some open problems} The above examples show that one can deduce singular value inequalities if
         an adequate description of extreme points is at hand. In general, however, such a catalog is not available, and for instance, 
         we cannot apply our technique to prove (or disprove) the following problems concerning partial traces (denoted by Tr$_1$).
         
         \bigskip
         
         \noindent {\bf Question 1.} Let $A, B$ are $n \times n$ Hermitian matrices. Does it follow that 
            $$ \left\|\mbox{Tr}_1 \: [(A \otimes I + I \otimes A)(B \otimes I - I \otimes B)] \right\|_{(k)} \leq 2  \sigma_1(A) \left\|\mbox{Tr}_1 \: (B \otimes I - I \otimes B) \right\|_{(k)}$$
          for all $1 \leq k \leq n?$  
            
        \bigskip
         
         \noindent {\bf Question 2.} Let $A, B$ are $n \times n$ Hermitian matrices. Does it follow that 
            $$ \left\|\mbox{Tr}_1 \: [(A \otimes I + I \otimes A)(B \otimes I - I \otimes B)] \right\|_{(k)} \leq 2 \sum_{i=1}^k \sigma_i(A) \sigma_i(\mbox{Tr}_1 \: (B \otimes I - I \otimes B))$$
          for all $1 \leq k \leq n?$  
          
          \bigskip
                
          It is known that if the Hermitians $A$ and $B$ are commuting then the previous inequalities hold as it can be seen in \cite{L}, \cite{Lr}.

\end{document}